\documentclass{amsart}%
\usepackage{amsfonts}
\usepackage{amsmath}
\usepackage{amssymb}
\usepackage{graphicx}%
\newtheorem{theorem}{Theorem}
\theoremstyle{plain}

\newtheorem{corollary}{Corollary}

\newtheorem{lemma}{Lemma}

\newtheorem{proposition}{Proposition}
\newtheorem{remark}{Remark}

\numberwithin{equation}{section}
\renewcommand{\leq}{\leqslant}
\renewcommand{\geq}{\geqslant}
\begin{document}
\title[A New Characterization of the CR Sphere]{A New Characterization of the CR Sphere and the sharp eigenvalue estimate for
the Kohn Laplacian}

\author{Song-Ying Li}
\address{Department of Mathematics, University of California, Irvine, CA 92697}
\email{sli@math.uci.edu}
\author{Duong Ngoc Son}
\address{Department of Mathematics, University of California, Irvine, CA 92697}
\email{snduong@math.uci.edu }
\author{Xiaodong Wang}
\address{Department of Mathematics, Michigan State University, East Lansing, MI 48824}
\email{xwang@math.msu.edu}
\maketitle

\section{Introduction\bigskip}

Let $\left(  M,\theta\right)  $ be a pseudohermitian
manifold of dimension $2m+1$ and $T$ the Reeb vector field. We always work
with a local unitary frame $\{T_{\alpha}:\alpha=1,\cdots,m\}$ for
$T^{1,0}\left(  M\right)  $ and its dual frame $\left\{  \theta^{\alpha
}\right\}  $. Thus%
\[
d\theta=\sqrt{-1}\sum_{\alpha}\theta^{\alpha}\wedge\overline{\theta^{\alpha}%
}.
\leqno(1.1)
\]
We will often denote $T$ by $T_{0}$. In \cite{LW}, the first and third named
authors proved the following Obata type result in CR geometry.

\begin{theorem}
\label{Ob}Let $M$ be a closed pseudohermitian manifold of dimension $2m+1\geq
5$. Suppose there is a real-valued nonzero function $u\in C^{\infty}\left(  M\right)
$ satisfying%
\begin{align*}
u_{\alpha,\beta}  &  =0,\\
u_{\alpha,\overline{\beta}}  &  =\left(  -\frac{\kappa}{2\left(  m+1\right)
}u+\frac{\sqrt{-1}}{2}u_{0}\right)  \delta_{\alpha\beta},
\end{align*}
for some constant $\kappa>0$. Then $M$ is equivalent to the sphere
$\mathbb{S}^{2m+1}$ with its standard pseudohermitian structure, up to a scaling.
\end{theorem}

A weaker version of the above theorem is also proved in dimension $3$ ($m=1$)
in \cite{LW} which requires an additional condition.

In this paper we prove a variant of the above theorem which characterizes the
CR sphere in terms of the existence of a (non-trivial) \emph{complex-valued} function satisfying 
a certain overdetermined system. The precise statement is the following theorem.

\begin{theorem}
\label{main}Let $\left(  M,\theta\right)  $ be closed pseudohermitian
manifold with dimension $2m+1\geq5$. Suppose\ that there exists a nonzero
complex-valued function $f$ on $M$ satisfying%
\begin{align*}
f_{\overline{\alpha},\overline{\beta}}  &  =0,\\
f_{\overline{\alpha},\beta}  &  =-cf\delta_{\alpha\beta}%
\end{align*}
for some constant $c>0$. Then $\left(  M,\theta\right)  $ is CR equivalent to
$\mathbb{S}^{2m+1}$ with its standard pseudohermitian structure, up to a scaling.
\end{theorem}

Theorem \ref{main} is motivated by the recent sharp lower bound for positive
eigenvalues of the Kohn Laplacian by Chanillo,
Chiu and Yang  \cite{CCY}, just as Theorem \ref{Ob} is motivated by
Greenleaf's sharp estimate for the first eigenvalue of the sublaplacian
$\Delta_{b}$. Recall that the Kohn Laplacian on a (complex-valued) function $f$ is
defined by%
\[
\square_{b}f=\overline{\partial}_{b}^{\ast}\overline{\partial}_{b}%
f=-f_{\overline{\alpha},\alpha}. \leqno(1.2)
\]
and its conjugate $\overline{\square}_{b}f=-f_{\alpha,\overline{\alpha}%
}=\square_{b}f-\sqrt{-1}mTf$. We have
\[
-\Delta_{b}=\square_{b}+\overline{\square}_{b}=2\square_{b}+\sqrt
{-1}mT=2\overline{\square}_{b}-\sqrt{-1}mT.\leqno(1.3)
\]
On a closed pseudohermitian manifold $M$, the Kohn Laplacian $\square_{b}$
defines a nonnegative self-adjoint operator on the Hilbert space $L^{2}\left(
M\right)  $ of all complex-valued functions $f$ with $|f|^2$ is integrable on $M$, and  the inner product on $L^2$ is defined by %
\[
\left\langle f_{1},f_{2}\right\rangle =\int_{M}f_{1}\overline{f}_{2}.\leqno(1.4)
\]
But unlike $\Delta_{b}$, it does not satisfy the H\"{o}rmander condition and, as a
result, its resolvent is not compact. The three dimensional case is more complicated than higher dimensions. Nevertheless, it is proved by
Burns and Epstein \cite{BE} that the spectrum of $\square_{b}$ in $\left(
0,\infty\right)  $ consists of point eigenvalues of finite multiplicity. In
general, there may exist a sequence of eigenvalues rapidly decreasing to zero.
Zero is an isolated eigenvalue iff the range of $\square_{b}$ is closed.

Motivated by the embedding problem for $3$-dimensional CR manifolds, Chanillo,
Chiu and Yang  \cite{CCY} recently proved the following eigenvalue estimate for the Kohn Laplacian:

\begin{theorem}
\label{3d}Let $M$ be a closed $3$-dimensional pseudohermitian manifold. If
the Paneitz operator $P_{0}$ is non-negative and the scalar curvature
$R\geq\kappa$, with $\kappa$ being a positive constant, then any nonzero
eigenvalue of $\square_{b}$ satisfies
\[
\lambda\geq\frac{1}{2}\kappa.
\]
\end{theorem}

Recall that the Paneitz operator $P_{0}$ :$C^{\infty}\left(  M\right)
\rightarrow C^{\infty}\left(  M\right)  $ on a closed pseudo-hermitian
manifold is defined by%
\[
P_{0}f=\left(  P_{\alpha}f\right)  _{,\overline{\alpha}}=f_{\overline{\gamma
},\gamma\alpha\overline{\alpha}}+m\sqrt{-1}\left(  A_{\alpha\beta}%
f_{\overline{\beta}}\right)  _{,\overline{\alpha}}.\leqno(1.5)
\]
We say that $P_{0}$ is non-negative if for any $f\in C^{\infty}\left(
M\right)  $%
\[
\int_{M}\overline{f}P_{0}f\geq0. \leqno(1.6)
\]
Though Chanillo, Chiu and Yang only proved the eigenvalue estimate for
$3$-dimensional pseudohermitian manifolds, their argument can be easily
generalized to higher dimensions. In fact, since the Paneitz operator $P_{0}$
is always non-negative on closed pseudohermitian manifolds of dimension
$\geqslant 5$, the statement is even simpler (see Chang and Wu \cite{CW}).
\begin{theorem}
\label{hd}\bigskip Let $\left(  M,\theta\right)  $be a closed pseudohermitian
manifold of dimension $2m+1$. Suppose for all $X\in T^{1,0}\left( M\right) $
\[
Ric\left(  X,X\right)  \geq\kappa\left\vert X\right\vert ^{2},
\]
where $\kappa$ is a positive constant. Then any nonzero eigenvalue of
$\square_{b}$ satisfies
\[
\lambda\geq\frac{m}{m+1}\kappa.
\]
\end{theorem}
Note that the estimate is sharp as equality holds
on the sphere $$\mathbb{S}^{2m+1}=\left\{  z\in\mathbb{C}^{m+1}:\left\vert
z\right\vert =1\right\}  $$ with the standard pseudohermitian structure%
\[
\theta_{0}=\left(  2\sqrt{-1}\,\overline{\partial}\left\vert z\right\vert
^{2}\right)  |_{\mathbb{S}^{2m+1}}.
\]

A natural question is whether the equality case characterizes the CR sphere with the standard pseudohermitian structure up to a scaling.
In their preprint \cite{CW} Chang and Wu studied this problem and proved
various partial results. One of them states that $M$ is indeed equivalent to
the CR sphere $\mathbb{S}^{2m+1}$ if equality holds in Theorem \ref{hd},
provided that the following identity
\[
\int_{M}A_{\overline{ \alpha}\overline{ \beta} } f_{\alpha}\overline{f}_{\beta}=0\leqno(1.7)
\]
is satisfied for a corresponding eigenfunction $f$. 

As a corollary of Theorem \ref{main}, we can resolve this question in the
general case. Namely, we have the following rigidity result.

\begin{corollary}
\label{eq}If equality holds in Theorem \ref{hd}, then $\left(  M,\theta
\right)  $ is equivalent to the CR sphere $\left(  \mathbb{S}^{2m+1}%
,\theta_{0}\right)  $, up to a scaling, i.e. there exists a CR diffeomorphism $F:M\rightarrow
\mathbb{S}^{2m+1}$ such that $F^{\ast}\theta_{0}=c\theta$ for some constant $c>0$.
\end{corollary}

We expect that a similar version of Theorem \ref{main} is true in dimension
$3$ from which the characterization of the equality case in Theorem \ref{3d}
would follow. But we have not been able to prove it yet. This is due to an 
additional difficulty that arises only in $3$-dimensional case: It is not clear when
functions annihilated by the CR Paneitz 
operator $P_0$ are CR-pluriharmonic.

Another remark is that despite the similarity between
these theorems and Obata's theorem in Riemannian
geometry, the proofs are essentially different due to the torsion of the Tanaka-Webster connection. In fact, a crucial step in the proof of Theorem~\ref{main} is to show that the torsion must vanish. But unlike the approach in
\cite{LW}, where the vanishing of torsion was deduced from estimates regarding high powers of
the real-valued function $u$ (or an eigenfunction of $\Delta_b$),
the vanishing of the torsion in our proof of Theorem~\ref{main} 
is proved by deriving various identities that are satisfied simultaneously only if the torsion is zero.

The paper is organized as follows. In Section 2, we review some basic facts in
CR geometry and discuss the Bochner formula for the Kohn Laplacian. In Section 3 we discuss the spectral theory of the Kohn Laplacian.
In Section 4 we discuss the eigenvalue estimate of Chanillo, Chiu and Yang and its generalization to higher dimensions. We deduce
Corollary \ref{eq} from Theorem \ref{main}. The proof of Theorem \ref{main} is
then presented in Section~5. 

{\bf Acknowledgement:}
We would like to thank Professor Mei-Chi Shaw for very instructive conversations on the analysis of the Kohn Laplacian on CR manifolds.

\section{Preliminaries}

We first review some basic facts in CR geometry. Define the operator
$P:C^{\infty}\left(  M\right)  \rightarrow\mathcal{A}^{1,0}\left(  M\right)  $
by%
\[
Pf=\left(  f_{\overline{\gamma},\gamma\alpha}+m\sqrt{-1}A_{\alpha\beta
}f_{\overline{\beta}}\right)  \theta^{\alpha}.
\]
We write $P_{\alpha}f=\left(  f_{\overline{\gamma},\gamma\alpha}+m\sqrt
{-1}A_{\alpha\beta}f_{\overline{\beta}}\right)  $. We also write%
\[
B_{\alpha\overline{\beta}}f=f_{\alpha,\overline{\beta}}-\frac{1}{m}%
f_{\gamma,\overline{\gamma}}\delta_{\alpha\beta}.
\]
The Paneitz operator $P_{0}$ :$C^{\infty}\left(  M\right)  \rightarrow
C^{\infty}\left(  M\right)  $ is defined by%
\[
P_{0}f=\left(  P_{\alpha}f\right)  _{,\overline{\alpha}}=f_{\overline{\gamma
},\gamma\alpha\overline{\alpha}}+m\sqrt{-1}\left(  A_{\alpha\beta}%
f_{\overline{\beta}}\right)  _{,\overline{\alpha}}.
\]
Graham-Lee \cite{GL} proved that $P_{0}$ is a real operator. Moreover, $P_{0}$
is symmetric, i.e. for $f_{1},f_{2}\in C^{\infty}\left(  M\right)  $ with one
of them compactly supported%
\[
\int_{M}P_{0}f_{1}\overline{f_{2}}=\int_{M}f_{1}\overline{P_{0}f_{2}}.
\]
Here, the integrals are taken with respect
to the volume form $dV = \theta \wedge (d\theta)^{m}$. They also proved the following identity when $M$ is closed:
\begin{equation}
\int_{M}\left\vert B_{\overline{\alpha},\beta}f\right\vert ^{2}=\int
_{M}\left\vert B_{\alpha,\overline{\beta}}\overline{f}\right\vert ^{2}%
=\frac{m-1}{m}\int_{M}\overline{f}P_{0}f. \label{P0}%
\end{equation}
Therefore, on a closed pseudohermitian manifold of dimension $2m+1\geqslant 5$, the
Paneitz operator is nonnegative, in the sense that for any complex-valued function~$f$, it holds that%
\[
\int_{M}\overline{f}P_{0}f\geq0.
\]
But in dimension $3$, there are closed pseudohermitian manifolds whose Paneitz
operator is NOT nonnegative. %Rossi's example 

The following Bochner-type formula for the Kohn Laplacian was established by
Chanillo, Chiu and Yang \cite{CCY} (see also Chang and Wu \cite{CW}).

\begin{proposition}
Let $f$ be a complex-valued function. Then
\begin{align*}
-\square_{b}\left\vert \overline{\partial}f\right\vert ^{2}  &  =\left(
\left\vert f_{\overline{\alpha},\overline{\beta}}\right\vert ^{2}+\left\vert
f_{\overline{\alpha},\beta}\right\vert ^{2}\right)  -\frac{m+1}{m}\left(
\square_{b}f\right)  _{\overline{\alpha}}\overline{f}_{\alpha}-\frac{1}%
{m}f_{\overline{\alpha}}\overline{\left(  \square_{b}f\right)  }_{\alpha}\\
&  +R_{\alpha\overline{\beta}}f_{\overline{\alpha}}\overline{f}_{\beta}%
-\frac{1}{m}\overline{f}_{\alpha}\overline{P_{\alpha}\overline{f}}+\frac
{m-1}{m}f_{\overline{\alpha}}\left(  P_{\alpha}\overline{f}\right)  .
\end{align*}
\end{proposition}
Integrating over a closed $M$ yields%
\begin{align*}
0  &  =\int_{M}\left(  \left\vert f_{\overline{\alpha},\overline{\beta}%
}\right\vert ^{2}+\left\vert f_{\overline{\alpha},\beta}\right\vert
^{2}\right)  -\frac{m+1}{m}\int_{M}\left(  \square_{b}f\right)  _{\overline
{\alpha}}\overline{f}_{\alpha}-\frac{1}{m}\int_{M}f_{\overline{\alpha}%
}\overline{\left(  \square_{b}f\right)  }_{\alpha}\\
&  +\int_{M}R_{\alpha\overline{\beta}}f_{\overline{\alpha}}\overline{f}%
_{\beta}-\frac{1}{m}\overline{f}_{\alpha}\overline{P_{\alpha}\overline{f}%
}+\frac{m-1}{m}f_{\overline{\alpha}}\left(  P_{\alpha}\overline{f}\right) \\
&  =\int_{M}\left(  \left\vert f_{\overline{\alpha},\overline{\beta}%
}\right\vert ^{2}+\left\vert B_{\overline{\alpha},\beta}f\right\vert
^{2}\right)  +\frac{1}{m}\int_{M}\left\vert \square_{b}f\right\vert ^{2}%
-\frac{m+1}{m}\int_{M}\left(  \square_{b}f\right)  _{\overline{\alpha}%
}\overline{f}_{\alpha}-\frac{1}{m}\int_{M}f_{\overline{\alpha}}\overline
{\left(  \square_{b}f\right)  }_{\alpha}\\
&  +\int_{M}R_{\alpha\overline{\beta}}f_{\overline{\alpha}}\overline{f}%
_{\beta}+\frac{m-2}{m}\int_{M}f_{\overline{\alpha}}\left(  P_{\alpha}%
\overline{f}\right) \\
&  =\int_{M}\left(  \left\vert f_{\overline{\alpha},\overline{\beta}%
}\right\vert ^{2}+\left\vert B_{\overline{\alpha},\beta}f\right\vert
^{2}\right)  -\frac{m+1}{m}\int_{M}\left\vert \square_{b}f\right\vert ^{2}\\
&  +\int_{M}R_{\alpha\overline{\beta}}f_{\overline{\alpha}}\overline{f}%
_{\beta}+\frac{m-2}{m}\int_{M}f_{\overline{\alpha}}\left(  P_{\alpha}%
\overline{f}\right)
\end{align*}
where we have used the decomposition $f_{\overline{\alpha},\beta}%
=B_{\overline{\alpha},\beta}f+\frac{1}{m}\square_{b}f\delta_{\alpha\beta}$.
Therefore, we have%
\begin{equation}
\frac{m+1}{m}\int_{M}\left\vert \square_{b}f\right\vert ^{2}=\int
_{M}\left\vert f_{\overline{\alpha},\overline{\beta}}\right\vert
^{2}+\left\vert B_{\overline{\alpha},\beta}f\right\vert ^{2}+R_{\alpha
\overline{\beta}}f_{\overline{\alpha}}\overline{f}_{\beta}+\frac{m-2}{m}%
\int_{M}f_{\overline{\alpha}}\left(  P_{\alpha}\overline{f}\right)  .
\label{fund0}%
\end{equation}
Integrating by parts yields%
\begin{align*}
\int_{M}f_{\overline{\alpha}}\left(  P_{\alpha}\overline{f}\right)   &
=-\int_{M}f\left(  P_{\alpha}\overline{f}\right)  _{,\overline{\alpha}}\\
&  =-\int_{M}fP_{0}\overline{f}\\
&  =-\int_{M}\overline{f}P_{0}f,
\end{align*}
In the last step, we have used the fact that $P_{0}$ is a real operator.
Plugging this identity and (\ref{P0}) into (\ref{fund0}) yields

\begin{proposition}
\label{intK}Let $M$ be a closed pseudohermitian manifold of dimension $2m+1$
and $f$ a complex function on $M$. Then%
\begin{equation}
\frac{m+1}{m}\int_{M}\left\vert \square_{b}f\right\vert ^{2}=\int
_{M}\left\vert f_{\overline{\alpha},\overline{\beta}}\right\vert ^{2}+\int
_{M}R_{\alpha\overline{\beta}}f_{\overline{\alpha}}\overline{f}_{\beta}%
+\frac{1}{m}\int_{M}\overline{f}P_{0}f. \label{fund}%
\end{equation}

\end{proposition}

\section{The spectral theory of the Kohn Laplacian}
Based on the work of Beals and Greiner \cite{BG},
Burn and Epstein \cite{BE} proved the following theorem in dimension 3.
\begin{theorem}
Let $M$ be a closed pseudohermitian manifold of dimension $3$.
The $\mathrm{spec}\left(  \square_{b}\right)  $ in $\left(  0,\infty\right)  $
consists of point eigenvalues of finite multiplicity. Moreover all these
eigenfunctions are smooth.
\end{theorem}

The spectral theory of the Kohn Laplacian in the higher dimensional case is in fact simpler. 
This is because  the Hodge theory for $(0,1)$-forms is valid for all 
closed pseudohermitian manifold of dimension $2m+1\geq 5$ by the fundamental work of Kohn \cite{K}. 
The spectral theory for the Kohn Laplacian can then be deduced from
the Hodge decomposition theorem for $(0,1)$-forms. This is known to the experts. 
But since it is not easily accessible in the
literature, we give a detailed presentation, using the Bochner formula as a short cut. 
For background and a detailed exposition of the Kohn theory we refer to the book \cite{CS} by 
Chen and Shaw, which is our primary source.

\begin{proposition}
\label{intK}Let $M$ be a closed pseudohermitian manifold of dimension
$2m+1\geq5$ and $f$ a complex function on $M$. Then%
\[
\frac{1}{m\left(  m-1\right)  }\int_{M}\left\vert \square_{b}f\right\vert
^{2}=\int_{M}\left\vert f_{\overline{\alpha},\overline{\beta}}\right\vert
^{2}+\frac{1}{m-1}\left\vert f_{\overline{\alpha},\beta}\right\vert ^{2}%
+\int_{M}R_{\sigma\overline{\alpha}}f_{\overline{\sigma}}\overline{f}_{\alpha
}.
\]
\end{proposition}

\begin{proof}
When $m\geq2$, we have by (\ref{P0})%
\begin{align*}
\frac{1}{m}\int_{M}\overline{f}P_{0}f  &  =\frac{1}{m-1}\int_{M}\left\vert
B_{\overline{\alpha},\beta}f\right\vert ^{2}\\
&  =\frac{1}{m-1}\int_{M}\left(  \left\vert f_{\overline{\alpha},\beta
}\right\vert ^{2}-\frac{1}{m}\left\vert \square_{b}f\right\vert ^{2}\right)  .
\end{align*}
Plugging this identity into (\ref{fund}) yields the desired identity.
\end{proof}

\bigskip

Throughout the rest of this section, we assume $m\geq2$. We have from the previous Proposition
\begin{equation}
\left\Vert \square_{b}f\right\Vert ^{2}\geq m\left(  m-1\right)  \left\Vert
f_{\overline{\alpha},\overline{\beta}}\right\Vert ^{2}+m\left\Vert
f_{\overline{\alpha},\beta}\right\Vert ^{2}-C\left\Vert \overline{\partial
}_{b}f\right\Vert ^{2},\label{fe}%
\end{equation}
where $C\geq0$ depends on the pseudohermitian Ricci tensor.

For $f\in Dom\left(  \square_{b}\right)  $ we have%
\[
\left\Vert \overline{\partial}_{b}f\right\Vert ^{2}=\left\langle \square
_{b}f,f\right\rangle \leq\left\Vert f\right\Vert \left\Vert \square
_{b}f\right\Vert.
\]
We will use this inequality implicitly.

\bigskip

\bigskip Let $\mathcal{H}$ denote the space of $L^{2}$ CR holomorphic
functions, i.e.
\[
\mathcal{H}=\left\{  f\in L^{2}\left(  M\right)  :\overline{\partial}%
_{b}f=0\right\}  .
\]

\begin{proposition}
\label{dbar}The range of $\overline{\partial}_{b}:L^{2}\left(  M\right)
\rightarrow L_{\left(  0,1\right)  }^{2}\left(  M\right)  $ is closed and more
precisely $R\left(  \overline{\partial}_{b}\right)  =\overline{\partial}%
_{b}\overline{\partial}_{b}^{\ast}\left(  Dom\left(  \square_{b}^{0,1}\right)
\right)  $. Moreover, for all $\beta\in R\left(  \overline{\partial}%
_{b}\right)  $, there exists a unique $f\in\mathcal{H}^{\bot}\cap Dom\left(
\overline{\partial}_{b}\right)  $ such that $\overline{\partial}_{b}f=\beta$ and%
\[
\left\Vert f\right\Vert \leq C\left\Vert \beta\right\Vert .
\]

\end{proposition}

\begin{proof}
The first part is Corollary 8.4.11 in \cite{CS}. Suppose $\beta
=\overline{\partial}_{b}u$. By the Hodge decomposition for $\overline
{\partial}_{b}$ on $L_{\left(  0,1\right)  }^{2}\left(  M\right)  $ (Theorem
8.4.10 in \cite{CS}) there exists $\alpha\in L_{\left(  0,1\right)  }%
^{2}\left(  M\right)  $ satisfying $\overline{\partial}_{b}^{\ast}\overline
{\partial}_{b}\alpha=0$ and
\[
\beta=\overline{\partial}_{b}\overline{\partial}_{b}^{\ast}\alpha.
\]
Moreover $\left\Vert \alpha\right\Vert _{1}\leq C\left\Vert \beta\right\Vert
$. It is easy to check that $f:=\overline{\partial}_{b}^{\ast}\alpha$ has all
the desired properties.
\end{proof}

\begin{proposition}
\label{boxb}The range of $\square_{b}:L^{2}\left(  M\right)  \rightarrow
L^{2}\left(  M\right)  $ is closed and more precisely $R\left(  \square
_{b}\right)  =\mathcal{H}^{\bot}$. Moreover, for all $\phi\in\mathcal{H}^{\bot}$,
there exists a unique $f\in\mathcal{H}^{\bot}\cap Dom\left(  \square
_{b}\right)  $ such that $\square_{b}f=\phi$ and%
\begin{equation}
\left\Vert f\right\Vert \leq C\left\Vert \overline{\partial}_{b}f\right\Vert .
\label{dbc}%
\end{equation}

\end{proposition}

\begin{proof}
Clearly $R\left(  \square_{b}\right)  \subset\mathcal{H}^{\bot}$. Suppose
$\phi=\square_{b}u\in R\left(  \square_{b}\right)  $. By Proposition
\ref{dbar}, there is a unique $f\in\mathcal{H}^{\bot}$ such that $\overline
{\partial}_{b}f=\overline{\partial}_{b}u$ and $\left\Vert f\right\Vert \leq
C\left\Vert \overline{\partial}_{b}f\right\Vert $. Then $\square_{b}%
f=\square_{b}u=\phi$. We now prove that $R\left(  \square_{b}\right)  $ is
closed. Suppose $\phi_{k}=\square_{b}f_{k}\rightarrow\phi$ in $L_{\left(
0,1\right)  }^{2}\left(  M\right)  $, with each $f_{k}\in\mathcal{H}^{\bot}$.
Then for $k<l$, we have $\square_{b}\left(  f_{k}-f_{l}\right)  =\phi_{k}%
-\phi_{l}$. Thus,%
\begin{align*}
\left\Vert \overline{\partial}_{b}f_{k}-\overline{\partial}_{b}f_{l}%
\right\Vert ^{2}  &  \leq\left\Vert \phi_{k}-\phi_{l}\right\Vert \left\Vert f_{k}-f_{l}%
\right\Vert \\
&  \leq C\left\Vert \phi_{k}-\phi_{l}\right\Vert \left\Vert \overline
{\partial}_{b}f-\overline{\partial}_{b}f_{l}\right\Vert .
\end{align*}
It follows that $\left\Vert \overline{\partial}_{b}f_{k}-\overline{\partial}%
_{b}f_{l}\right\Vert \leq C\left\Vert \phi_{k}-\phi_{l}\right\Vert $. Applying
(\ref{dbc}) again yields $\left\Vert f_{k}-f_{l}\right\Vert \leq
C^{2}\left\Vert \phi_{k}-\phi_{l}\right\Vert $. Therefore, $\left\{
f_{k}\right\}  $ is Cauchy in $L^{2}\left(  M\right)  $. Denote its limit by
$f$. Then $f_{k}\rightarrow f$ and $\square_{b}f_{k}\rightarrow\phi$. As
$\square_{b}$ is a closed operator, we conclude that $f\in Dom\left(
\square_{b}\right)  $ and $\square_{b}f=\phi$. Therefore, $R\left(  \square
_{b}\right)  $ is closed.

If $R\left(  \square_{b}\right)  $ was not the entire $\mathcal{H}^{\bot}$,
then there exists a nonzero $\phi\in\mathcal{H}^{\bot}$ that is perpendicular
to $R\left(  \square_{b}\right)  $. This implies $\phi\in\mathcal{H}$,
obviously a contradiction.
\end{proof}

\bigskip

Therefore, the operator $\square_{b}:\mathcal{H}^{\bot}\cap Dom\left(
\square_{b}\right)  \rightarrow\mathcal{H}^{\bot}$ is bijective. The
inverse operator exists and is denoted by $T:\mathcal{H}^{\bot}\rightarrow\mathcal{H}^{\bot
}$. Namely, for each $\phi\in\mathcal{H}^{\bot}$, we define $f=T\left(
\phi\right)\in \mathcal{H}^{\bot}  $ to be the unique solution to $\square_{b}f=\phi$ 
(which exists and unique by Proposition \ref{boxb}).

\begin{theorem}
The operator $T:\mathcal{H}^{\bot}\rightarrow\mathcal{H}^{\bot}$ is compact.
\end{theorem}

\begin{proof}
By (\ref{fe}) there exists a constant $C>0$ such that for any $f\in C^{\infty
}\left(  M\right)  $%
\[
\left\Vert f_{\overline{\alpha},\overline{\beta}}\right\Vert ^{2}+\left\Vert
f_{\overline{\alpha},\beta}\right\Vert ^{2}\leq C\left(  \left\Vert
\overline{\partial}_{b}f\right\Vert ^{2}+\left\Vert \square_{b}f\right\Vert
^{2}\right)  .
\]
It follows, by the H\"{o}rmander estimate (see Theorem 8.2.5 in \cite{CS}), that
\begin{equation}
\left\Vert \overline{\partial}_{b}f\right\Vert _{1/2}^{2}\leq C\left(
\left\Vert \overline{\partial}_{b}f\right\Vert ^{2}+\left\Vert \square
_{b}f\right\Vert ^{2}\right), \label{half}%
\end{equation}
where $\left\Vert\cdot \right\Vert _{1/2}$ is the norm for the Sobolev space $W_{\left(  0,1\right)  }^{1/2}\left(  M\right)$.
By approximation, this inequality holds for all $f\in Dom\left(  \square
_{b}\right)  $. We can further assume that $f\in\mathcal{H}^{\bot}$.

Suppose $\left\{  f_{k}\right\}  \subset\mathcal{H}^{\bot}\cap Dom\left(
\square_{b}\right)  $ is a sequence such that $\phi_{k}=\square_{b}f_{k}$ are
bounded in $L^{2}\left(  M\right)  $. By (\ref{half}), $\overline{\partial}%
_{b}f_{k}$ are bounded in $W_{\left(  0,1\right)  }^{1/2}\left(  M\right)  $. 
By the Sobolev embedding theorem and
passing to a subsequence, we can assume that $\left\{
\overline{\partial}_{b}f_{k}\right\}  $ is Cauchy in $L_{\left(  0,1\right)
}^{2}\left(  M\right)  $. By (\ref{dbc}), $\left\{  f_{k}\right\}  $ is Cauchy
in $L^{2}\left(  M\right)  $. The proof is complete.
\end{proof}
\begin{theorem}
The $\mathrm{spec}\left(  \square_{b}\right)  $ consists of countably many eigenvalues
$\lambda_{0}=0<\lambda_{1}<\lambda_{2}<\cdots$ with $\lambda_{i}\rightarrow\infty$ as $i\rightarrow\infty$. Moreover, for $i\geq1$, each
$\lambda_{i}$ is an eigenvalue of finite multiplicity and all the
corresponding eigenfunctions are smooth.
\end{theorem}

\begin{proof}
We have proved, in Proposition, \ref{boxb} that $R\left(  \square_{b}\right)  $
is closed. Thus, $\lambda_{0}=0$ is an eigenvalue whose corresponding
eigenspace is $\mathcal{H}$, which is of infinite dimensional. With respect to the
orthogonal decomposition $L^{2}\left(  M\right)  =\mathcal{H}\oplus
\mathcal{H}^{\bot}$, the operator $\lambda I-\square_{b}$ is given by the following
matrix%
\[%
\begin{bmatrix}
\lambda I & 0\\
0 & \lambda I-\square_{b}|_{\mathcal{H}^{\bot}}%
\end{bmatrix}
.
\]
Therefore, $\lambda>0$ is in $\mathrm{spec}\left(  \square_{b}\right)  $ if and only if
$\lambda^{-1}$ is in $\mathrm{spec}\left(  T\right)  $. As $T$ is compact, $\mathrm{spec}\left(
T\right)  \cap\left(  0,+\infty\right)  $ consists of countably many
eigenvalues of finite multiplicities $\mu_{1}>\mu_{2}>\cdots$ $\ $\ with
$\lim\mu_{i}=0$. Therefore, $\mathrm{spec}\left(  \square_{b}\right)  $ consists of
$\lambda_{0}=0<\lambda_{1}<\lambda_{2}<\cdots$ with $\lambda_{i}=1/\mu_{i}$
for all $i\geq1$ and the eigenspace of $\square_{b}$ corresponding to
$\lambda_{i}$ equals the eigenspace of $T$ corresponding to $\mu_{i}$.

Suppose $f$ is an eigenfunction corresponding to an eigenvalue $\lambda>0$, i.e. $\square
_{b}f=\lambda f$. Then the $\left(  0,1\right)  $-form $\beta=\overline
{\partial}_{b}f$ satisfies%
\begin{align*}
\square_{b}\beta & =\overline{\partial}_{b}\square_{b}f\\
& =\lambda\overline{\partial}_{b}f\\
& =\lambda\beta.
\end{align*}
By the Hodge theory for $\left(  0,1\right)  $-forms, $\beta$ is smooth. As
$f=\frac{1}{\lambda}\overline{\partial}_{b}^{\ast}\beta$, we see that $f$ is
smooth. 
\end{proof}

\section{The eigenvalue estimate}
With the spectral theory of $\square_{b}$ understood, we can now state the following

\begin{theorem}
\label{est}Let $M$ be a closed pseudohermitian manifold of dimension $2m+1$.
When $m=1$, we further assume that the Paneitz operator is non-negative.
Suppose%
\[
Ric\left(  X,X\right)  \geq\kappa\left\vert X\right\vert ^{2},
\]
where $\kappa$ is a positive constant. Then any nonzero eigenvalue of
$\square_{b}$ satisfies
\[
\lambda\geq\frac{m}{m+1}\kappa.
\]
\end{theorem}
\begin{proof}
Suppose $f$ is a nonzero eigenfunction with eigenvalue $\lambda>0$. By
(\ref{fund}), under the assumptions%
\begin{align*}
\frac{m+1}{m}\lambda^{2}\int_{M}\left\vert f\right\vert ^{2}  &  =\int
_{M}\left\vert f_{\overline{\alpha},\overline{\beta}}\right\vert ^{2}+\int
_{M}R_{\alpha\overline{\beta}}f_{\overline{\alpha}}\overline{f}_{\beta}%
+\frac{1}{m}\int_{M}\overline{f}P_{0}f\\
&  \geq\kappa\int_{M}\left\vert \overline{\partial}f\right\vert ^{2}\\
&  =\kappa\int_{M}\overline{f}\square_{b}f\\
&  =\lambda\kappa\int_{M}\left\vert f\right\vert ^{2}.
\end{align*}
Thus, $\lambda\geq\frac{m}{m+1}\kappa$.
\end{proof}

\bigskip Theorem~\ref{est} was first proved by Chanillo, Chiu and Yang \cite{CCY} in the
case $m=1$. We have followed basically the same argument (see also Chang and
Wu \cite{CW}).

%To discuss the equality case, we can always assume that $\kappa=\left(
%m+1\right)  /2$ by scaling.

\begin{proposition}
\label{ids}Suppose $\lambda=\frac{m}{m+1}\kappa$ in Theorem \ref{est} and $f$
a corresponding eigenfunction. Then we must have:
\begin{enumerate}
\renewcommand{\labelenumi}{(\roman{enumi})}
\item If $m=1$, then%
\[
f_{\overline{1},\overline{1}}=0,\quad f_{\overline{1},1}=-\frac{\kappa}{2} %
f,\quad P_{0}f=0;
\]
\item If $m\geq2$, then%
\[
f_{\overline{\alpha},\overline{\beta}}=0,\quad f_{\overline{\alpha},\beta}%
=-\frac{\kappa}{m+1}f\delta_{\alpha\beta}.
\]
\end{enumerate}
\end{proposition}
\begin{proof}
If equality holds, by inspecting the proof of Theorem \ref{est}, we must have
$f_{\overline{\alpha},\overline{\beta}}=0$ and $\int_{M}\overline{f}P_{0}f=0$.
As $P_{0}$ is nonnegative, it follows easily that $P_{0}f=0$. As $P_{0}$ is a
real operator, we also have $P_{0}\overline{f}=0$. When $m\geq2$, this implies
by (\ref{P0}) that
\begin{align*}
f_{\overline{\alpha},\beta}  &  =-\frac{1}{m}\square_{b}f\delta_{\alpha\beta}\\
&  =-\frac{\kappa}{m+1}f\delta_{\alpha\beta}.\qedhere
\end{align*}
\end{proof}

Combining this Proposition and Theorem \ref{main}, we immediately obtain the following 
\begin{corollary}
\bigskip Suppose $\lambda=\frac{m}{m+1}\kappa$ in Theorem \ref{est} and
$m\geq 2$. Then $\left(  M,\theta\right)  $ is CR equivalent to $\mathbb{S}%
^{2m+1}$ with its standard pseudohermitian structure, up to scaling.
\end{corollary}

\section{\bigskip Proof of the Main Theorem}
We now prove our main theorem (Theorem \ref{main}).
By scaling, we may assume $c=1/2$. Theorem \ref{main} is equivalent to the following
\begin{theorem}
Let $\left(  M,\theta\right)  $ be closed pseudohermitian
manifold with dimension $2m+1\geq5$. Suppose\ that there exists a nonzero
complex function $f$ on $M$ satisfying%
\begin{align*}
f_{\overline{\alpha},\overline{\beta}}  &  =0,\\
f_{\overline{\alpha},\beta}  &  =-\frac{1}{2}f\delta_{\alpha\beta}%
\end{align*}
Then $\left(  M,\theta\right)  $ is CR equivalent to
the $\mathbb{S}^{2m+1}$ with its standard pseudohermitian structure.
\end{theorem}

Therefore we have
\begin{align}
f_{\overline{\alpha},\overline{\beta}}  &  =0,\label{i1}\\
f_{\overline{\alpha},\beta}  &  =-\frac{1}{2}f\delta_{\alpha\beta}. \label{i2}%
\end{align}
Using (\ref{i1}) and (\ref{i2}) it is easy to derive%
\begin{equation}
\left(  \left\vert \overline{\partial}_{b}f\right\vert ^{2}\right)  _{\alpha
}=-\frac{1}{2}f\overline{f}_{\alpha}. \label{gr2}%
\end{equation}

\begin{proposition}
We have%
\begin{align}
A_{\alpha\beta}\overline{f}_{\gamma}  &  =A_{\alpha\gamma}\overline{f}_{\beta
},\label{i4}\\
f_{\gamma}  &  =2\sqrt{-1}A_{\gamma\sigma}f_{\overline{\sigma}}. \label{i5}%
\end{align}

\end{proposition}

\begin{proof}
Differentiating (\ref{i1}) yields%
\[
A_{\alpha\beta}\overline{f}_{\gamma}=A_{\alpha\gamma}\overline{f}_{\beta}.
\]
Differentiating (\ref{i2}) yields%
\begin{align*}
-\delta_{\alpha\beta}f_{\gamma}/2  &  =f_{\overline{\alpha},\beta\gamma}\\
&  =f_{\overline{\alpha},\gamma\beta}-\sqrt{-1}\left(  \delta_{\alpha\beta
}A_{\gamma\sigma}-\delta_{\alpha\gamma}A_{\beta\sigma}\right)  f_{\overline
{\sigma}}\\
&  =-\delta_{\alpha\gamma}f_{\beta}/2-\sqrt{-1}\left(  \delta_{\alpha\beta
}A_{\gamma\sigma}-\delta_{\alpha\gamma}A_{\beta\sigma}\right)  f_{\overline
{\sigma}}.
\end{align*}
Hence%
\[
\delta_{\alpha\beta}\left(  f_{\gamma}/2-\sqrt{-1}A_{\gamma\sigma}%
f_{\overline{\sigma}}\right)  =\delta_{\alpha\gamma}\left(  f_{\beta}%
/2-\sqrt{-1}A_{\beta\sigma}f_{\overline{\sigma}}\right)  .
\]
Therefore $f_{\gamma}/2-\sqrt{-1}A_{\gamma\sigma}f_{\overline{\sigma}}=0$.
\end{proof}

Let $Q=\sqrt{-1}A_{\alpha\beta}f_{\overline{\alpha}}f_{\overline{\beta}}$.
Set
\[
K=\left\{  p\in M:\overline{\partial}_{b}f\left(  p\right)  =0\right\}  .
\]
On $M\backslash K$ we define
\[
\psi=2Q/\left\vert \overline{\partial}_{b}f\right\vert ^{2}.
\]
Note that $\psi$ is smooth and bounded on $M\backslash K$.

\begin{lemma}
$M\backslash K$ is open and dense.
\end{lemma}

\begin{proof}
We need to prove that $K$ has no interior point. Write $f=u+\sqrt{-1}v$ with
$u$ and $v$ real. Then, using (\ref{i2})
\begin{align*}
u_{\alpha,\overline{\beta}}  &  =\frac{1}{2}\left(  f_{\alpha,\overline{\beta
}}+\overline{f}_{\alpha,\overline{\beta}}\right) \\
&  =\frac{1}{2}\left(  f_{\overline{\beta},\alpha}-\sqrt{-1}f_{0}%
\delta_{\alpha\beta}+\overline{f}_{\alpha,\overline{\beta}}\right) \\
&  =-\frac{1}{4}\left(  f+2\sqrt{-1}f_{0}+\overline{f}\right)  \delta
_{\alpha\beta}.
\end{align*}
Therefore, $u$ is CR pluriharmonic. Similarly, $v$ is also CR pluriharmonic$.$

Now suppose $\overline{\partial}_{b}f=0$ on a connected open set $U$. By
(\ref{i2}) $f=0$ on $U$. Hence, $u$ and $v$ both vanish on $U$. Being CR
pluriharmonic, $u$ and $v$ then must be identically zero on $M$. This is a contradiction.
\end{proof}

\begin{proposition}
On $M\backslash K$ we have%
\begin{align}
A_{\alpha\beta}  &  =-\frac{\sqrt{-1}\psi}{2\left\vert \overline{\partial}%
_{b}f\right\vert ^{2}}\overline{f}_{\alpha}\overline{f}_{\beta}%
,\label{torsion}\\
f_{\gamma}  &  =\psi\overline{f}_{\gamma}. \label{prop}%
\end{align}

\end{proposition}

\begin{proof}
Using (\ref{i4}), we compute%
\begin{align*}
A_{\alpha\beta}\left\vert \overline{\partial}_{b}f\right\vert ^{2}  &
=A_{\alpha\gamma}\overline{f}_{\beta}f_{\overline{\gamma}}\\
&  =A_{\gamma\alpha}\left\vert \overline{\partial}_{b}f\right\vert ^{2}%
\frac{\overline{f}_{\beta}f_{\overline{\gamma}}}{\left\vert \overline
{\partial}_{b}f\right\vert ^{2}}\\
&  =A_{\gamma\sigma}\overline{f}_{\alpha}f_{\overline{\sigma}}\frac
{\overline{f}_{\beta}f_{\overline{\gamma}}}{\left\vert \overline{\partial}%
_{b}f\right\vert ^{2}}\\
&  =-\frac{\sqrt{-1}Q}{\left\vert \overline{\partial}_{b}f\right\vert ^{2}%
}\overline{f}_{\alpha}\overline{f}_{\beta}\\
&  =-\sqrt{-1}{\psi\over 2} \overline{f}_{\alpha}\overline{f}_{\beta}.
\end{align*}
This proves (\ref{torsion}). To prove (\ref{prop}), we compute using
(\ref{i5}) and (\ref{torsion})%
\begin{align*}
f_{\gamma}  &  =2\sqrt{-1}A_{\gamma\sigma}f_{\overline{\sigma}}\\
&  =\frac{\psi}{\left\vert \overline{\partial}_{b}f\right\vert ^{2}}%
\overline{f}_{\gamma}\overline{f}_{\sigma}f_{\overline{\sigma}}\\
&  =\psi\overline{f}_{\gamma}.\qedhere
\end{align*}

\end{proof}

\begin{remark}
From (\ref{torsion}), we obtain on $M\backslash K$%
\begin{align*}
\left\vert A\right\vert ^{2}  &  :=\sum_{\alpha,\beta}\left\vert
A_{\alpha\beta}\right\vert ^{2}\\
&  =\frac{1}{4}\left\vert \psi\right\vert ^{2}.
\end{align*}
Therefore, $\left\vert \psi\right\vert ^{2}$ extends smoothly to the entire $M$.
\end{remark}

\begin{proposition}
\label{pbar}On $M\backslash K$ we have%
\begin{equation}
\overline{\partial}_{b}\psi=0 \label{scr}%
\end{equation}
and
\begin{equation}
\sqrt{-1}f_{0}-\frac{1}{2}f+\frac{1}{2}\psi\overline{f}=0. \label{f0}%
\end{equation}

\end{proposition}

\begin{proof}
Differentiating $f_{\alpha}=\psi\overline{f}_{\alpha}$ and using (\ref{i2})
yields%
\begin{align*}
f_{\alpha,\overline{\beta}}  &  =\psi_{\overline{\beta}}\overline{f}_{\alpha
}+\psi\overline{f}_{\alpha,\overline{\beta}}\\
&  =\psi_{\overline{\beta}}\overline{f}_{\alpha}-\frac{1}{2}\psi\overline
{f}\delta_{\alpha\beta}.
\end{align*}
By using (\ref{i2}) again, we further compute the left hand side%
\begin{align*}
f_{\alpha,\overline{\beta}}  &  =f_{\overline{\beta},\alpha}+\sqrt{-1}%
f_{0}\delta_{\alpha\beta}\\
&  =-\frac{1}{2}f\delta_{\alpha\beta}+\sqrt{-1}f_{0}\delta_{\alpha\beta}.
\end{align*}
Therefore,%
\[
\psi_{\overline{\beta}}\overline{f}_{\alpha}=\left(  \sqrt{-1}f_{0}-\frac
{1}{2}f+\frac{1}{2}\psi\overline{f}\right)  \delta_{\alpha\beta}.
\]
From this, it follows easily (since $m\geq2$) that $\sqrt{-1}f_{0}-f/2+\psi
\overline{f}/2=0$ and $\psi_{\overline{\beta}}=0$.
\end{proof}

\begin{proposition}
\bigskip We have%
\begin{equation}
R_{\alpha\overline{\beta}}f_{\overline{\alpha}}=\frac{m+1}{2}f_{\overline
{\beta}}. \label{i3}%
\end{equation}

\end{proposition}

\begin{proof}
Using (\ref{i1}) and (\ref{i2}), we compute%
\begin{align*}
0  &  =\overline{f}_{\alpha,\beta\overline{\gamma}}\\
&  =\overline{f}_{\alpha,\overline{\gamma}\beta}+\sqrt{-1}\delta_{\beta\gamma
}\overline{f}_{\alpha,0}-R_{\beta\overline{\gamma}\alpha\overline{\sigma}%
}\overline{f}_{\sigma}\\
&  =-\frac{1}{2}\overline{f}_{\beta}\delta_{\alpha\gamma}+\sqrt{-1}%
\delta_{\beta\gamma}\left(  \overline{f}_{0,\alpha}-A_{\alpha\sigma}%
\overline{f}_{\overline{\sigma}}\right)  -R_{\beta\overline{\gamma}%
\alpha\overline{\sigma}}\overline{f}_{\sigma}.
\end{align*}
Differentiating (\ref{f0}) and using (\ref{scr}) and (\ref{prop}) yields%
\begin{align*}
\sqrt{-1}\overline{f}_{0,\alpha}  &  =\frac{1}{2}\left(  \overline{\psi
}f_{\alpha}-\overline{f}_{\alpha}\right) \\
&  =\frac{1}{2}\left(  \left\vert \psi\right\vert ^{2}-1\right)  \overline
{f}_{\alpha}.
\end{align*}
Plugging this into the previous equation and using (\ref{prop}) again, we
obtain%
\begin{align*}
0  &  =-\frac{1}{2}\overline{f}_{\beta}\delta_{\alpha\gamma}-R_{\beta
\overline{\gamma}\alpha\overline{\sigma}}\overline{f}_{\sigma}+\left[
\frac{1}{2}\left(  \left\vert \psi\right\vert ^{2}-1\right)  \overline
{f}_{\alpha}-\sqrt{-1}\overline{\psi}A_{\alpha\sigma}f_{\overline{\sigma}%
}\right]  \delta_{\beta\gamma}\\
&  =-\frac{1}{2}\overline{f}_{\beta}\delta_{\alpha\gamma}-R_{\beta
\overline{\gamma}\alpha\overline{\sigma}}\overline{f}_{\sigma}-\frac{1}%
{2}\overline{f}_{\alpha}\delta_{\beta\gamma},
\end{align*}
where in the last step, we have used (\ref{torsion}). Therefore,
\[
-R_{\beta\overline{\gamma}\alpha\overline{\sigma}}\overline{f}_{\sigma}%
=\frac{1}{2}\left(  \overline{f}_{\beta}\delta_{\alpha\gamma}+\overline
{f}_{\alpha}\delta_{\beta\gamma}\right)  .
\]
Taking trace over $\beta$ and $\gamma$ yields (\ref{i3}).
\end{proof}

\bigskip

Since the Paneitz operator $P_{0}$ is real, we have%
\[
\int_{M}fP_{0}\overline{f}=\overline{\int_{M}\overline{f}P_{0}f}=0.
\]
Applying the Bochner formula to $\overline{f}$ yields%
\begin{align*}
\frac{m+1}{m}\int_{M}\left\vert \square_{b}\overline{f}\right\vert ^{2}  &
=\int_{M}\left\vert f_{\alpha,\beta}\right\vert ^{2}+\int_{M}R_{\alpha
\overline{\beta}}\overline{f}_{\overline{\alpha}}f_{\beta}+\frac{1}{m}\int
_{M}\overline{f}P_{0}f\\
&  =\int_{M}\left\vert f_{\alpha,\beta}\right\vert ^{2}+\frac{m+1}{2}\int
_{M}\left\vert \psi\right\vert ^{2}\left\vert \overline{\partial}%
_{b}f\right\vert ^{2}.
\end{align*}
We compute on $M\backslash K$, using (\ref{prop}) and (\ref{i1})%
\begin{align}
f_{\alpha,\beta}  &  =\psi_{\beta}\overline{f}_{\alpha}+\psi\overline
{f}_{\alpha,\beta}\label{fab}\\
&  =\psi_{\beta}\overline{f}_{\alpha}.\nonumber
\end{align}
From this, we get on $M\backslash K$%
\begin{equation}
\left\vert \partial_{b}\psi\right\vert ^{2}\left\vert \overline{\partial}%
_{b}f\right\vert ^{2}=\sum_{\alpha,\beta}\left\vert f_{\alpha\beta}\right\vert
^{2}. \label{fab2}%
\end{equation}
Notice that the right hand side is smooth on $M$. Therefore, $\left\vert \partial
_{b}\psi\right\vert ^{2}\left\vert \overline{\partial}_{b}f\right\vert ^{2}$
extends smoothly to the entire $M$ and the above inequality holds on $M$.
Similarly, using (\ref{scr}) as well%
\[
\square_{b}\overline{f}=-\overline{f}_{\overline{\alpha},\alpha}=-\left(
\overline{\psi}f_{\overline{\alpha}}\right)  _{,\alpha}=-\overline{\psi
}f_{\overline{\alpha},\alpha}=\frac{m}{2}\overline{\psi}f.
\]
Plugging these into the integral identity, we obtain%
\begin{align*}
&  \frac{m\left(  m+1\right)  }{4}\int_{M}\left\vert \psi\right\vert
^{2}\left\vert f\right\vert ^{2}\\
&  =\int_{M}\left\vert \partial_{b}\psi\right\vert ^{2}\left\vert
\overline{\partial}_{b}f\right\vert ^{2}+\frac{m+1}{2}\int_{M}\left\vert
\psi\right\vert ^{2}\left\vert \overline{\partial}_{b}f\right\vert ^{2},
\end{align*}
i.e.%
\begin{equation}
\int_{M}\left\vert \partial_{b}\psi\right\vert ^{2}\left\vert \overline
{\partial}_{b}f\right\vert ^{2}=\frac{m+1}{2}\int_{M}\left\vert \psi
\right\vert ^{2}\left(  \frac{m}{2}\left\vert f\right\vert ^{2}-\left\vert
\overline{\partial}_{b}f\right\vert ^{2}\right)  . \label{int0}%
\end{equation}

\bigskip

\begin{lemma}
\label{pb}We have on $M\backslash K$%
\[
\partial_{b}\psi=0.
\]

\end{lemma}

\begin{proof}
By (\ref{fab}), we have $\psi_{\alpha}\overline{f}_{\beta}=\psi_{\beta
}\overline{f}_{\alpha}$. Therefore, on $M\backslash K$%
\begin{equation}
\psi_{\alpha}=\left(  \sum_{\beta}\psi_{\beta}f_{\overline{\beta}}\right)
\overline{f}_{\alpha}\left\vert \overline{\partial}_{b}f\right\vert ^{-2}.
\label{sa}%
\end{equation}
From this, we get%
\[
\left\vert \partial_{b}\psi\right\vert ^{2}\left\vert \overline{\partial}%
_{b}f\right\vert ^{2}=\left\vert \sum_{\beta}\psi_{\beta}f_{\overline{\beta}%
}\right\vert ^{2}.
\]
Since $\overline{\partial}_{b}\psi=0$, we have $\psi_{\alpha,\overline{\beta}%
}=\sqrt{-1}\psi_{0}\delta_{\alpha\beta}$. For each $\alpha$ we compute using
(\ref{sa}) and Proposition \ref{ids}%
\begin{align*}
\psi_{\alpha\overline{\alpha}}  &  =\left[  \left(  \sum_{\beta}\psi
_{\beta,\overline{\alpha}}f_{\overline{\beta}}\right)  \overline{f}_{\alpha
}+\left(  \sum_{\beta}\psi_{\beta}f_{\overline{\beta}}\right)  \overline
{f}_{\alpha,\overline{\alpha}}\right]  \left\vert \overline{\partial}%
_{b}f\right\vert ^{-2}\\
&  -\left(  \sum_{\beta}\psi_{\beta}f_{\overline{\beta}}\right)  \overline
{f}_{\alpha}\left\vert \overline{\partial}_{b}f\right\vert ^{-4}\sum_{\gamma
}f_{\overline{\gamma}}\overline{f}_{\gamma,\overline{\alpha}}\\
&  =\left[  \psi_{\alpha,\overline{\alpha}}\left\vert f_{\overline{\alpha}%
}\right\vert ^{2}-\frac{1}{2}\left(  \sum_{\beta}\psi_{\beta}f_{\overline
{\beta}}\right)  \overline{f}\right]  \left\vert \overline{\partial}%
_{b}f\right\vert ^{-2}+\frac{1}{2}\left(  \sum_{\beta}\psi_{\beta}%
f_{\overline{\beta}}\right)  \overline{f}\left\vert f_{\overline{\alpha}%
}\right\vert ^{2}\left\vert \overline{\partial}_{b}f\right\vert ^{-4}.
\end{align*}
Hence,%
\[
\psi_{\alpha\overline{\alpha}}\left(  1-\left\vert \overline{\partial}%
_{b}f\right\vert ^{-2}\left\vert f_{\overline{\alpha}}\right\vert ^{2}\right)
=-\frac{1}{2}\left(  \sum_{\beta}\psi_{\beta}f_{\overline{\beta}}\right)
\overline{f}\left\vert \overline{\partial}_{b}f\right\vert ^{-2}\left(
1-\left\vert \overline{\partial}_{b}f\right\vert ^{-2}\left\vert
f_{\overline{\alpha}}\right\vert ^{2}\right)  .
\]
It follows that on $M\backslash K$%
\[
\psi_{\alpha\overline{\alpha}}=-\frac{1}{2}\left(  \sum_{\beta}\psi_{\beta
}f_{\overline{\beta}}\right)  \left\vert \overline{\partial}_{b}f\right\vert
^{-2}%
\]
Set
\[
B=\sum_{\beta}\left\vert \psi\right\vert _{\beta}^{2}f_{\overline{\beta}}.
\]
Note that $B$ is a smooth function on $M$. Then on $M\backslash K$, as
$\overline{\partial}_{b}\psi=0$
\begin{equation}
\overline{\psi}\psi_{\alpha\overline{\alpha}}=-\frac{1}{2}\left(  \sum_{\beta
}\overline{\psi}\psi_{\beta}f_{\overline{\beta}}\right)  \left\vert
\overline{\partial}_{b}f\right\vert ^{-2}=-\frac{1}{2}B\overline{f}\left\vert
\overline{\partial}_{b}f\right\vert ^{-2}. \label{saba}%
\end{equation}
\ We compute on $M\backslash K$, using Proposition \ref{ids} and (\ref{saba})
\begin{align*}
&  \left(  \overline{f}\left\vert \psi\right\vert ^{2}f_{\overline{\alpha}%
}+\left\vert \overline{\partial}_{b}f\right\vert ^{2}\left\vert \psi
\right\vert _{\overline{\alpha}}^{2}\right)  _{,\alpha}\\
&  =\left(  \overline{f}\left\vert \psi\right\vert ^{2}f_{\overline{\alpha}%
}+\left\vert \overline{\partial}_{b}f\right\vert ^{2}\psi\overline{\psi
}_{\overline{\alpha}}\right)  _{,\alpha}\\
&  =\overline{f}B+\left\vert \psi\right\vert ^{2}\left\vert \overline
{\partial}_{b}f\right\vert ^{2}-\frac{m}{2}\left\vert \psi\right\vert
^{2}\left\vert f\right\vert ^{2}-f\psi\overline{f}_{\alpha}\overline{\psi
}_{\overline{\alpha}}+\left\vert \overline{\partial}_{b}f\right\vert
^{2}\left(  \psi\overline{\psi}_{\overline{\alpha},\alpha}+\left\vert
\partial_{b}\psi\right\vert ^{2}\right) \\
&  =\frac{1}{2}\left(  \overline{f}B-f\overline{B}\right)  +\left\vert
\psi\right\vert ^{2}\left(  \left\vert \overline{\partial}_{b}f\right\vert
^{2}-\frac{m}{2}\left\vert f\right\vert ^{2}\right)  +\left\vert
\overline{\partial}_{b}f\right\vert ^{2}\left\vert \partial_{b}\psi\right\vert
^{2}.
\end{align*}
Since both sides are smooth on $M$ and $M\backslash K$ is open and dense, the
above identity holds on the entire $M$. Integrating over $M$ \ and taking the
real part yields%
\[
\int_{M}\left\vert \overline{\partial}_{b}f\right\vert ^{2}\left\vert
\partial_{b}\psi\right\vert ^{2}=\int_{M}\left\vert \psi\right\vert
^{2}\left(\frac{m}{2}\left\vert f\right\vert ^{2}-\left\vert \overline{\partial}%
_{b}f\right\vert ^{2}\right)  .
\]
Combining this identity with (\ref{int0}), we obtain $\int_{M}\left\vert
\overline{\partial}_{b}f\right\vert ^{2}\left\vert \partial_{b}\psi\right\vert
^{2}=0$. Therefore, $\partial_{b}\psi=0$ on $M\backslash K$.
\end{proof}

\begin{lemma}
$\psi=0$ and therefore, the torsion vanishes.
\end{lemma}

\begin{proof}
By Proposition \ref{pbar} and Proposition \ref{pb}, $\overline{\partial}%
_{b}\psi=0$ and $\partial_{b}\psi=0$ on $M\backslash K$. Therefore, $\psi$ is locally
constant on $M\backslash K$. Since $\left\vert \psi\right\vert ^{2}$ extends
smoothly to $M$, $\left\vert \psi\right\vert ^{2}$ is constant on $M$.
Differentiating (\ref{torsion}) on $M\backslash K$ using Proposition \ref{ids}
and (\ref{gr2}), we get on $M\backslash K$%
\[
2A_{\alpha\beta,\gamma}\left\vert \overline{\partial}_{b}f\right\vert
^{2}=A_{\alpha\beta}f\overline{f}_{\gamma}.
\]
Hence,%
\begin{align*}
2A_{\alpha\beta,\gamma}f_{\overline{\alpha}}f_{\overline{\beta}}%
f_{\overline{\gamma}}\left\vert \overline{\partial}_{b}f\right\vert ^{2}  &
=A_{\alpha\beta}f_{\overline{\alpha}}f_{\overline{\beta}}f\left\vert
\overline{\partial}_{b}f\right\vert ^{2}\\
&  =-\sqrt{-1}Qf\left\vert \overline{\partial}_{b}f\right\vert ^{2}.%
\end{align*}
Thus, on $M\backslash K$, we have $Qf=2\sqrt{-1}A_{\alpha\beta,\gamma
}f_{\overline{\alpha}}f_{\overline{\beta}}f_{\overline{\gamma}}$, i.e.%
\[
\psi f=4\sqrt{-1}A_{\alpha\beta,\gamma}f_{\overline{\alpha}}f_{\overline
{\beta}}f_{\overline{\gamma}}/\left\vert \overline{\partial}_{b}f\right\vert
^{2}.
\]
From this, we obtain (first on $M\backslash K$ and then, by continuity, on the
whole $M$ as $\left\vert \psi\right\vert $ is continuous on $M$ and
$M\backslash K$ is open and dense)
\begin{equation}
\left\vert \psi\right\vert \left\vert f\right\vert \leq4C\left\vert
\overline{\partial}_{b}f\right\vert , \label{sifi}%
\end{equation}
where, $C=\max_{M}\sqrt{\sum\left\vert A_{\alpha\beta,\gamma}\right\vert ^{2}}$.

Let $p_{0}\in M$ be a point where $\left\vert f\right\vert ^{2}$ achieves its
maximum. Suppose $\overline{\partial}_{b}f\left(  p_{0}\right)  \neq0$, i.e
$p_{0}\in M\backslash K$. Then differentiating at $p_{0}$ and using
(\ref{prop}), we have%
\begin{align*}
0  &  =f\overline{f}_{\alpha}+f_{\alpha}\overline{f}\\
&  =\overline{f}_{\alpha}\left(  f+\psi\overline{f}\right)  .
\end{align*}
Hence,
\begin{equation}
\psi\left(  p_{0}\right)  =-\frac{f\left(  p_{0}\right)  }{\overline{f\left(
p_{0}\right)  }}. \label{sp0}%
\end{equation}
As $\left\vert \psi\right\vert $ is constant, we have $\left\vert
\psi\right\vert \equiv1$. By (\ref{f0}) we have $\sqrt{-1}f_{0}=\frac{1}%
{2}\left(  f-\psi\overline{f}\right)  $on $M\backslash K$. Differentiating and
using (\ref{prop}) yields%
\begin{align*}
\sqrt{-1}f_{0\alpha}  &  =\frac{1}{2}\left(  f_{a}-\psi\overline{f}_{\alpha
}\right)  =0,\\
\sqrt{-1}f_{0\overline{\alpha}}  &  =\frac{1}{2}\left(  f_{\overline{a}}%
-\psi\overline{f}_{\overline{\alpha}}\right) \\
&  =\frac{1}{2}(1-\left\vert \psi\right\vert ^{2})f_{\overline{\alpha}}\\
&  =0.
\end{align*}
Therefore, $f_{0}$ is constant. As $\int_{M}f_{0}=0$, we must have $f_{0}%
\equiv0$. Then (\ref{f0}) reduces to $f=\psi\overline{f}$. At $p_{0}$ this
yields $\psi\left(  p_{0}\right)  =f\left(  p_{0}\right)  /\overline{f\left(
p_{0}\right)  }$. This is contradictory to (\ref{sp0}). Therefore,
$\overline{\partial}_{b}f\left(  p_{0}\right)  =0$, then the inequality
(\ref{sifi}) implies $\left\vert \psi\left(  p_{0}\right)  \right\vert =0$.
Consequently, $\psi\equiv0$.
\end{proof}

\bigskip Since $\psi\equiv0$, we have, in view of (\ref{torsion}) (\ref{prop})
and (\ref{f0})
\begin{align*}
A_{\alpha\beta}  &  =0,\\
\partial_{b}f  &  =0,\\
\sqrt{-1}f_{0}  &  =\frac{1}{2}f.
\end{align*}
Write $f=u+\sqrt{-1}v$ in terms of its real part $u$ and imaginary part $v$.
From the above identities and Proposition \ref{ids} it is easy to prove the following

\begin{proposition}
We have $v=2u_{0}$, while $u$ satisfies%
\begin{align*}
u_{\alpha\beta}  &  =0,\\
u_{\alpha,\overline{\beta}}  &  =\left(  -\frac{1}{4}u+\frac{\sqrt{-1}}%
{2}u_{0}\right)  \delta_{\alpha\beta}.
\end{align*}

\end{proposition}

With this proved, we can now apply the result of Li-Wang \cite{LW} (Theorem 5
therein) to conclude that $M$ is CR equivalent to $\mathbb{S}^{2m+1}$ with
its standard pseudohermitian structure. In fact, we do not need that result
in its full generality as we have the additional condition $A=0$ at our
disposal. Since the argument there under the additional condition $A=0$ is
simple, we give an outline for completeness. Let $g_{\theta}$ be the adapted
Riemannian metric and $D^{2}u$ the Riemannian Hessian. Then from the above
identities we obtain by standard calculation (see Proposition 2 in \cite{LW})%
\[
D^{2}u=-\frac{u}{4}g_{\theta},
\]
here we used the fact that the torsion $A=0$.

By Obata's theorem \cite{O}, $(M,g_{\theta})$ is isometric to the sphere
$\mathbb{S}^{2m+1}$ with the metric $g_{0}=4g_{c}$, where $g_{c}$ is the
canonical metric. Without loss of generality, we can take $(M,g_{\theta})$ to
be ($\mathbb{S}^{2m+1},g_{0}$). Then $\theta$ is a pseudohermitian structure
on $\mathbb{S}^{2m+1}$ whose adapted metric is $g_{0}$ and the associated
Tanaka-Webster connection is torsion-free. It is a well known fact that the
Reeb vector field $T$ is then a Killing vector field for $g_{0}$ (see Remark 1 in \cite{LW}). %this can be easily proved by the first formula in Remark \ref{one}) . 
Therefore there
exists a skew-symmetric matrix $A$ such that for all $X\in\mathbb{S}%
^{2m+1},T(X)=AX$, here we use the obvious identification between
$z=(z_{1},\ldots,z_{m+1})\in\mathbb{C}^{m+1}$ and $X=(x_{1},y_{1}%
,\ldots,x_{m+1},y_{m+1})\in\mathbb{R}^{2m+2}$. Changing coordinates by an
orthogonal transformation we can assume that $A$ is of the following form
\[
A=\left[
\begin{array}
[c]{ccc}%
\begin{array}
[c]{cc}%
0 & a_{1}\\
-a_{1} & 0
\end{array}
&  & \\
& \ddots & \\
&  &
\begin{array}
[c]{cc}%
0 & a_{m+1}\\
-a_{m+1} & 0
\end{array}
\end{array}
\right]
\]
where $a_{i}\geq 0$. Therefore
\[
T=\sum_{i}a_{i}\left(  y_{i}\frac{\partial}{\partial x_{i}}-x_{i}%
\frac{\partial}{\partial y_{i}}\right)
\]
Since $T$ is of unit length we must have
\[
4\sum_{i}a_{i}^{2}(x_{i}^{2}+y_{i}^{2})=1
\]
on $\mathbb{S}^{2m+1}$. Therefore all the $a_{i}$'s are equal to $1/2$. It
follows that
\[
\theta=g_{0}(T,\cdot)=2\sqrt{-1}\,\overline{\partial}(|z|^{2}-1).
\]
This finishes the proof of Theorem \ref{main}.

\end{document}